\input ./style/arxiv-vmsta.cfg
\documentclass[numbers,compress,v1.0.1]{vmsta}
\usepackage{vtexbibtags}

\volume{6}
\issue{3}
\pubyear{2019}
\firstpage{377}
\lastpage{395}
\aid{VMSTA140}
\doi{10.15559/19-VMSTA140}
\articletype{research-article}

\startlocaldefs
\newcommand{\lleft}{\left}
\newcommand{\rright}{\right}
\allowdisplaybreaks
\newtheorem{thm}{Theorem}
\newtheorem{lemma}{Lemma}
\newtheorem{cor}{Corollary}
\theoremstyle{definition}
\newtheorem{remark}{Remark}
\theoremstyle{definition}
\newcommand{\prob}{\mathsf P}
\newcommand{\topr}{\xrightarrow{\prob}}
\DeclareMathOperator{\cov}{Cov}

\newcommand{\rrvert}{\vert}
\newcommand{\llvert}{\vert}
\allowdisplaybreaks

\newcommand{\coloneqq}{:=}
\endlocaldefs

\begin{document}

\begin{frontmatter}
\pretitle{Research Article}

\title{Maximum likelihood estimation in the non-ergodic fractional Vasicek model}

\author{\inits{S.}\fnms{Stanislav}~\snm{Lohvinenko}\ead[label=e1]{stanislav.lohvinenko@gmail.com}}
\author{\inits{K.}\fnms{Kostiantyn}~\snm{Ralchenko}\thanksref{cor1}\ead[label=e2]{k.ralchenko@gmail.com}\orcid{0000-0001-7208-3130}}
\thankstext[type=corresp,id=cor1]{Corresponding author.}
\address{Department of Probability Theory, Statistics and Actuarial
Mathematics,
\institution{Taras~Shevchenko National University of~Kyiv},
64,~Volodymyrs'ka St., 01601~Kyiv,~\cny{Ukraine}}



\markboth{S. Lohvinenko, K. Ralchenko}{Maximum likelihood estimation in the non-ergodic fractional Vasicek model}

\begin{abstract}
We investigate the fractional Vasicek model described by the stochastic
differential equation $dX_{t}=(\alpha -\beta X_{t})\,dt+\gamma \,dB
^{H}_{t}$, $X_{0}=x_{0}$, driven by the fractional Brownian motion
$B^{H}$ with the known Hurst parameter $H\in (1/2,1)$. We study the maximum
likelihood estimators for unknown parameters $\alpha $ and $\beta $ in
the non-ergodic case (when $\beta < 0$) for arbitrary $x_{0}\in
\mathbb{R}$, generalizing the result of Tanaka, Xiao and Yu (2019) for
particular $x_{0}=\alpha /\beta $, derive their asymptotic distributions
and prove their asymptotic independence.
\end{abstract}

\received{\sday{27} \smonth{5} \syear{2019}}
\revised{\sday{7} \smonth{8} \syear{2019}}
\accepted{\sday{5} \smonth{9} \syear{2019}}
\publishedonline{\sday{23} \smonth{9} \syear{2019}}
\end{frontmatter}

\section{Introduction}%
\label{sec:introduction}
The present paper deals with the fractional Vasicek model\index{fractional Vasicek model} of the form
%
\begin{equation}
\label{eq:fr-vas} dX_{t} = ( \alpha - \beta X_{t} ) dt +
\gamma dB_{t}^{H}, \quad X_{0}=x_{0}
\in \mathbb{R},
\end{equation}
where $B^{H}$ is the fractional Brownian motion\index{fractional Brownian motion} with the Hurst index
$H \in (1/2, 1)$.\index{Hurst index} It is a generalization of the
classical interest rate model proposed by O. Vasicek
\cite{vasicek} in 1977. This generalization enables to study processes
with long-range dependence, which arise in financial mathematics and
several other areas such as telecommunication networks, investigation
of turbulence and image processing. In recent years, many articles on
various financial applications of the fractional Vasicek model
\eqref{eq:fr-vas} have appeared, see
e.g.~\cite{Chronopoulou-Viens12a,Corlay,Fink13,HLW14,Song-Li-2018,XZZC14}).
In order to use this model in practice, a theory of parameter estimation\index{parameter estimation}
is necessary.

Notice that in the particular case $\alpha =0$, \eqref{eq:fr-vas} is a
so-called fractional Ornstein--Uhlenbeck process, introduced in
\cite{CKM}. The drift parameter estimation\index{parameter estimation} for it has been studied since
2002, see the paper \cite{KB}, where the maximum likelihood
estimation was considered. The asymptotic and exact distributions of the
maximum likelihood estimator\index{maximum likelihood estimator (MLE)} (MLE) were investigated later in
\cite{Bercu2010,Tanaka2013,Tanaka2015}. Alternative approaches to the
drift parameter estimation\index{parameter estimation} were proposed and studied in
\cite{BESO,Bishwal11,HuNu,HuSong13,HuNuZhou,kumirase}. We refer to the
article \cite{MR-survey} for a survey on this topic, and to the
book \cite{KMR2017} for its more detailed presentation.

In the general case, the least squares and ergodic-type estimators of
unknown parameters $\alpha $ and $\beta $ were studied in
\cite{Lohvinenko-Ralchenko-Zhuchenko-2016,Xiao-Yu-2017,Xiao-Yu-2018}.
The corresponding MLEs of $\alpha $ and $\beta $ were presented in
\cite{Lohvinenko-Ralchenko-2017}. Their consistency and asymptotic
normality were proved there for the case $\beta >0$. Slightly more
general results were proved in \cite{Lohvinenko-Ralchenko-2018},
where joint asymptotic normality of MLE of the vector parameter
$(\alpha , \beta )$ was established. Recently Tanaka et al.
\cite{Tanaka-Xiao-Yu} investigated asymptotic behavior of MLEs\index{maximum likelihood estimator (MLE)} in the
cases $\beta =0$ and $\beta <0$. However, in the latter case the
asymptotic distribution was obtained only under assumption that
$x_{0}=\frac{\alpha }{\beta }$. The study of the case $x_{0}\neq \frac{
\alpha }{\beta }$ requires a different technique and still remains an open
problem. The goal of the present paper is to fill in the gap and to
derive asymptotic distributions of the MLEs\index{maximum likelihood estimator (MLE)} of $\alpha $ and
$\beta $ for arbitrary $x_{0}\in \mathbb{R}$, $\alpha \in \mathbb{R}$
and $\beta <0$. Moreover, we prove that the MLEs\index{maximum likelihood estimator (MLE)} for $\alpha $ and
$\beta $ are asymptotically independent.

The asymptotic behavior of the process $X$ and of the estimators
substantially depends on the sign of the parameter $\beta $. If
$\beta <0$, then the process $X$ behaves as $O_{\mathsf{P}}(e^{-
\beta T})$ as $T\to \infty $, hence it is non-ergodic. If $\beta >0$,
then $X_{T}=O_{\mathsf{P}}(1)$, as $T\to \infty $, and the process has
ergodic properties, see,
e.g.,~\cite{Lohvinenko-Ralchenko-Zhuchenko-2016}. The method for
the hypothesis testing of the sign of $\beta $ was developed in
\cite{Hypot}.

In this article we restrict ourselves to the case $\frac{1}{2}<H<1$. Our
proofs are based on the results of the papers \cite{KB} and
\cite{Lohvinenko-Ralchenko-2018}, which are valid only for $H\in (
\frac{1}{2},1)$ and cannot be immediately extended to the case
$H\in (0,\frac{1}{2})$. In particular, the integral representation
\eqref{eq:XviaS} below, which is the starting point for derivation of
moment generating functions\index{moment generating functions (MGF)} (MGFs) in Lemmas \ref{l:S_I_mgf} and
\ref{th:S_I_J_K_mgf}, holds for $H\in (0,\frac{1}{2})$ with different
(and more complicated) kernel $K_{H}$. Therefore, the asymptotic
behavior of the MLEs\index{maximum likelihood estimator (MLE)} in this case requires a separate study.

The paper is organized as follows. In Section
\ref{sec:model_description} we describe the model and the estimators,
and introduce the notation. Section \ref{sec:aux} contains the results
on distributions and asymptotic behavior of stochastic processes
involved into MLEs.\index{maximum likelihood estimator (MLE)} In Section~\ref{sec:main} we formulate and prove the
main results on asymptotic distributions of MLEs.\index{maximum likelihood estimator (MLE)} Some auxiliary facts
and results concerning modified Bessel functions of the first kind and
MGFs related to the normal distribution are collected in the appendices.

\section{Preliminaries}%
\label{sec:model_description}
Let $(\varOmega , \mathfrak{F},  \{  \mathfrak{F}_{t}  \}  ,
\mathsf{P})$ be a complete probability space with filtration. Let
$B^{H}=\{B^{H}_{t},\allowbreak t\geq 0\}$ be a fractional Brownian motion\index{fractional Brownian motion} on this
probability space, that is, a centered Gaussian process with the covariance
function
%
\begin{equation}
\label{eq:cov} \mathbb{E}B^{H}_{t}B^{H}_{s}
= \frac{1}{2} \bigl(t^{2H}+s^{2H}- \llvert t-s
\rrvert ^{2H} \bigr).
\end{equation}
It follows from \eqref{eq:cov} that $\mathbb{E}(B^{H}_{t}-B^{H}_{s})^{2}
= |t-s|^{2H}$. Hence, there exists a modification of $B^{H}$, which is
$\delta $-H\"{o}lder continuous for all $\delta \in (0,H)$.

We study the fractional Vasicek model,\index{fractional Vasicek model} described by the stochastic
differential equation
%
\begin{equation}
\label{eq:equation} X_{t}=x_{0}+\int_{0}^{t}
( \alpha -\beta X_{s} ) ds + \gamma B_{t}^{H},
\quad t\ge 0.
\end{equation}
The main goal is to estimate parameters $\alpha \in \mathbb{R}$ and
$\beta <0$ by continuous observations of a trajectory of $X$ on the
interval $[0,T]$. We assume that the parameters $\gamma >0$ and
$H\in (1/2,1)$ are known. This assumption is natural, because
$\gamma $ and $H$ can be obtained explicitly from the observations by
considering realized power variations,\index{power variations} see Remark~\ref{rem:est-gamma}
below.

Equation~\eqref{eq:equation} has a unique solution, which is given by
%
\begin{equation}
\label{eq:solution} X_{t}=x_{0} e^{-\beta t}+
\frac{\alpha }{\beta } \bigl(1-e^{-\beta t} \bigr) + \gamma \int
_{0}^{t} e^{-\beta (t-s)} dB_{s}^{H},
\quad t \ge 0,
\end{equation}
where $\int_{0}^{t} e^{-\beta (t-s)} dB_{s}^{H}$ is a path-wise
Riemann--Stieltjes integral. It exists due to~\cite[Proposition
A.1]{CKM}.

Following~\cite{KB2000}, for $0 < s < t \leq T$ we define
\begin{gather*}
\kappa _{H} = 2 H \varGamma (3/2 - H ) \varGamma (H + 1/2 ), \qquad
\lambda _{H} = \frac{2 H \varGamma (3 - 2 H) \varGamma (H + 1/2)}{\varGamma (3/2
- H)},
\\
k_{H}(t, s) = \kappa _{H}^{-1}
s^{1/2 - H} (t - s)^{1/2 - H}, \qquad w_{t}^{H}
= \lambda _{H}^{-1} t^{2 - 2 H}.
\end{gather*}
We introduce also three stochastic processes
\begin{align*}
P_{H}(t) &= \frac{1}{\gamma } \frac{d}{d w_{t}^{H}} \int
_{0}^{t} k _{H}(t, s)
X_{s} \, ds,
\\
Q_{H}(t) &= \frac{1}{\gamma }\frac{d}{d w_{t}^{H}} \int
_{0}^{t} k _{H}(t, s) (\alpha -
\beta X_{s}) \, ds,
\\
S_{t} &= \frac{1}{\gamma } \int_{0}^{t}
k_{H}(t, s) \, d X_{s}.
\end{align*}
Note that by~\cite[Lemma 4.1]{Lohvinenko-Ralchenko-2017}
%
\begin{equation}
\label{eq:Q_via_P} Q_{H}(t) = \frac{\alpha }{\gamma } - \beta
P_{H}(t).
\end{equation}
According to \cite[Theorem 1]{KB2000}, the process $S$ is an
$(\mathfrak{F}_{t})$-semimartingale with the decomposition
%
\begin{equation}
\label{eq:SviaM} S_{t} = \int_{0}^{t}
Q_{H}(s) \, dw_{s}^{H} +
M_{t}^{H},
\end{equation}
where $M^{H}_{t}=\int_{0}^{t} k_{H}(t,s)\,dB^{H}_{s}$ is a Gaussian
martingale\index{Gaussian martingale} with the variance function $\langle M^{H}\rangle =w^{H}$. The natural
filtrations of processes $S$ and $X$ coincide. Moreover, the process
$X$ admits the following representation
%
\begin{equation}
\label{eq:XviaS} X_{t} = \int_{0}^{t}
K_{H}(t, s) \, d S_{s},
\end{equation}
where $K_{H}(t, s) = \gamma H (2 H - 1) \int_{s}^{t} r^{H - 1/2} (r -
s)^{H - 3/2} \, dr$.

\begin{remark}
\label{rem:est-gamma}
If we observe the whole path $\{X_{t},t\in [0,T]\}$, then the parameters
$\gamma $ and $H$ can be obtained from observations explicitly in the
following way. Let $  \{  t_{i}^{(n)}   \}  $ be a partition of
$[0, T]$, such that $\sup_{i}  \llvert  t_{i+1}^{(n)} - t_{i}^{(n)}
 \rrvert  \to 0$, as $n \to \infty $. Denote $Z_{T} = \int_{0}^{T} k
_{H}(T, s) \, d X_{s}=\gamma S_{T}$. From \eqref{eq:SviaM} it follows
that $  \langle  Z   \rangle  _{T} = \gamma ^{2} w_{T}^{H}$ a.s.
Hence, the parameter $\gamma $ is calculated as the limit
\begin{equation*}
\label{eq:gamma-estimation} \gamma ^{2}= \bigl( w_{T}^{H}
\bigr)^{-1} \lim_{n} \sum
_{i} \left( Z _{t_{i + 1}^{(n)}} - Z_{t_{i}^{(n)}}
\right)^{2} \quad \text{a.s.}
\end{equation*}

The Hurst index $H$\index{Hurst index} can be evaluated as follows:
\begin{equation*}
H = \frac{1}{2} - \frac{1}{2}\lim_{n}
\log _{2} \left(\frac{\sum_{i=1}
^{2n-1}  \left(X_{t_{i + 1}^{(2n)}} - 2 X_{t_{i}^{(2n)}} + X_{t_{i-1}
^{(2n)}}  \right)^{2}}{\sum_{i=1}^{n-1}  \left(X_{t_{i + 1}^{(n)}} - 2 X
_{t_{i}^{(n)}} + X_{t_{i-1}^{(n)}}  \right)^{2}} \right) \quad \text{a.s.,}
\end{equation*}
see, e.g., \cite[Sec.~3.1]{KMR2017}. There exist several other
methods of the Hurst index evaluation based on power variations\index{power variations} of
$X$. We refer to the books~\cite{BLL2014,KMR2017} for further
information on this subject.
\end{remark}

Applying the analog of the Girsanov formula for a fractional Brownian
motion\index{fractional Brownian motion} (\cite[Theorem 3]{KB2000}, see also \cite{KMM}) and
\eqref{eq:SviaM}, one can obtain the likelihood ratio $\frac{d
\mathsf{P}_{\alpha ,\beta }(T)}{d\mathsf{P}_{0,0}(T)}$ for the
probability measure $\mathsf{P}_{\alpha ,\beta }(T)$ corresponding to
our model and the probability measure $\mathsf{P}_{0,0}(T)$ corresponding
to the model with zero drift~\cite{Lohvinenko-Ralchenko-2017}:
%
\begin{equation}
\label{eq:likelihood_ratio} %
\begin{aligned}[b] \frac{d\mathsf{P}_{\alpha ,\beta }(T)}{d\mathsf{P}_{0,0}(T)} &= \exp \Biggl
\{ \int_{0}^{T} Q_{H}(t) \, d
S_{t} - \frac{1}{2} \int_{0}
^{T} \bigl( Q_{H}(t) \bigr)^{2} \, d
w_{t}^{H} \Biggr\}
\\[-2pt]
&= \exp \lleft\{ \frac{\alpha }{\gamma } S_{T} - \beta \int
_{0}^{T} P _{H}(t) \, d
S_{t} - \frac{\alpha ^{2}}{2 \gamma ^{2}} w_{T}^{H}
\rright .
\\[-2pt]
&\quad + \lleft . \frac{\alpha \beta }{\gamma } \int_{0}^{T}
P_{H}(t) \, d w_{t}^{H} -
\frac{\beta ^{2}}{2} \int_{0}^{T} \bigl(
P_{H}(t) \bigr) ^{2} \, d w_{t}^{H}
\rright\} . \end{aligned} %
\end{equation}
MLEs\index{maximum likelihood estimator (MLE)} of parameters $\alpha $ and $\beta $ maximize
\eqref{eq:likelihood_ratio} and have the following
form~\cite{Lohvinenko-Ralchenko-2017}:
%
\begin{equation}
\label{eq:mle} \widehat{\alpha }_{T} = \frac{ S_{T} K_{T}- I_{T} J_{T}}{w_{T}^{H} K
_{T}-J_{T}^{2}}
\gamma , \qquad \widehat{\beta }_{T} = \frac{ S_{T} J_{T}-w_{T}^{H} I_{T} }{w_{T}^{H}
K_{T}-J_{T}^{2}},
\end{equation}
where
\begin{equation*}
I_{T} = \int_{0}^{T}
P_{H}(t) \, d S_{t}, \qquad J_{T} = \int
_{0}^{T} P _{H}(t) \, d
w_{t}^{H}, \qquad K_{T} = \int
_{0}^{T} \bigl( P_{H}(t) \bigr)
^{2} \, d w_{t}^{H}.
\end{equation*}
It is worth noting that using the definition of $P_{H}(t)$ one can
easily represent $J_{T}$ in the following way
\begin{equation*}
J_{T} = \frac{1}{\gamma } \int_{0}^{T}
k_{H}(T, s) X_{s} \, d s.
\end{equation*}

\section{Auxiliary results}%
\label{sec:aux}
In this section we find exact and asymptotic distributions of the
statistics $S_{T}$, $I_{T}$, $J_{T}$, $K_{T}$ and related random
variables\index{random variables} and vectors.

We start with the bivariate MGF\index{bivariate MGF} of the vector $(S_{T}, I_{T})$. For the case
$\beta >0$, it was derived in
\cite[Lemma~3.3]{Lohvinenko-Ralchenko-2018}. However, the same proof is
valid for the case $\beta <0$. The following result is a reformulation
of \cite[Lemma~3.3]{Lohvinenko-Ralchenko-2018}, obtained by applying
the formula \eqref{eq:even} from Appendix~\ref{sec:Bessel-function}.

\begin{lemma}
\label{l:S_I_mgf}
The moment generating function of $(S_{T}, I_{T})$ equals
\begin{align*}
m_{1}^{(\alpha ,\beta )}(\xi _{1}, \xi
_{2}) &= \mathbb{E} \bigl[ \exp \{ \xi _{1}
S_{T} + \xi _{2} I_{T} \} \bigr]
\\
&= D^{(\alpha ,\beta )}(\xi _{2})^{-\frac{1}{2}} \exp \Biggl\{
\frac{1}{8D
^{(\alpha ,\beta )}(\xi _{2})}\sum_{i=1}^{4}
A_{i}^{(\alpha ,\beta )}( \xi _{1},\xi
_{2}) - \frac{\xi _{2} T}{2} \Biggr\} ,
\end{align*}
where
%
\begin{align}
\label{eq:D_def} D^{(\alpha ,\beta )}(\xi _{2}) &= \biggl(1 -
\frac{\xi _{2}}{2 \beta } \biggr) ^{2} + \frac{\xi _{2}^{2}}{4 \beta ^{2}}
e^{-2\beta T} + \biggl( \frac{
\xi _{2}}{\beta } - \frac{\xi _{2}^{2}}{2 \beta ^{2}}
\biggr) \frac{(-
\beta ) \pi T}{4 \sin \pi H} e^{-\beta T}
\nonumber \\
&\quad \times \biggl[ I_{-H} \biggl( -\frac{\beta T}{2} \biggr)
I_{H - 1} \biggl( -\frac{\beta T}{2} \biggr) + I_{1 - H}
\biggl( -\frac{\beta
T}{2} \biggr) I_{H} \biggl( -
\frac{\beta T}{2} \biggr) \biggr],
\end{align}
%
\begin{align}
A_{1}^{(\alpha ,\beta )}(\xi _{1},\xi
_{2}) &= \xi _{2} \bigl(c_{1} \bigl(
\tfrac{\alpha }{\beta } \bigr)\xi _{1} - c_{2} \bigl(
\tfrac{\alpha
}{\beta } \bigr)\xi _{2} \bigr) (-\beta )^{H-1}
T^{1 - H} e^{-\frac{3
\beta T}{2}} I_{1 - H} \bigl( -
\tfrac{\beta T}{2} \bigr), \label{eq:A_1_def}
\\
A_{2}^{(\alpha ,\beta )}(\xi _{1},\xi _{2})
&= \bigl( \xi _{1}^{2} c _{3} - \xi
_{1} \xi _{2} c_{4} \bigl(
\tfrac{\alpha }{\beta } \bigr) + \xi _{2}^{2}
c_{5} \bigl(\tfrac{\alpha }{\beta } \bigr) \bigr) T^{2 - 2H}
e^{-\beta T}
\nonumber
\\*
&\quad \times I_{1 - H} \bigl( -\tfrac{\beta T}{2} \bigr)
I_{H - 1} \bigl( -\tfrac{\beta T}{2} \bigr), \label{eq:A_2_def}
\\
A_{3}^{(\alpha ,\beta )}(\xi _{1},\xi _{2})
&= \xi _{2} ( \xi _{2} - 2 \beta ) c_{6}
\bigl(\tfrac{\alpha }{\beta } \bigr) (-\beta )^{2H-1} T e^{-\beta T}
I_{1 - H} \bigl( -\tfrac{\beta T}{2} \bigr) I_{-H} \bigl(
-\tfrac{\beta T}{2} \bigr), \label{eq:A_3_def}
\\
A_{4}^{(\alpha ,\beta )}(\xi _{1},\xi _{2})
&= \bigl(c_{1} \bigl(\tfrac{
\alpha }{\beta } \bigr)\xi
_{1}-c_{2} \bigl(\tfrac{\alpha }{\beta } \bigr)\xi
_{2} \bigr) (\xi _{2} - 2\beta ) (-\beta
)^{H-1}
\nonumber
\\*
&\quad \times T^{1 - H} e^{-\frac{\beta T}{2}} I_{1 - H} \bigl( -
\tfrac{
\beta T}{2} \bigr), \label{eq:A_4_def}
\end{align}
\begin{align*}
c_{1} \bigl(\tfrac{\alpha }{\beta } \bigr) &= \bigl(
x_{0} - \tfrac{
\alpha }{\beta } \bigr) 4\rho _{H}, &
c_{4} \bigl(\tfrac{\alpha }{
\beta } \bigr) &= \bigl( x_{0}
- \tfrac{\alpha }{\beta } \bigr)\rho _{H} 2^{2H + 1} \varGamma
(H),
\\
c_{2} \bigl(\tfrac{\alpha }{\beta } \bigr) &= \bigl( x_{0}
- \tfrac{
\alpha }{\beta } \bigr)^{2} \frac{\lambda _{H}^{*} 2^{2H + 1} \rho _{H}
^{2}}{\varGamma (1 - H)}, &
c_{5} \bigl(\tfrac{\alpha }{\beta } \bigr) &= \bigl(
x_{0} - \tfrac{\alpha }{\beta } \bigr)^{2}
\frac{\lambda _{H}
^{*} 2^{4H - 1} \rho _{H}^{2} \varGamma (H)}{\varGamma (1 - H)},
\\
c_{3} &= \frac{2 \varGamma (H) \varGamma (1 - H)}{\lambda _{H}^{*}}, & c _{6} \bigl(
\tfrac{\alpha }{\beta } \bigr) &= \biggl( x_{0} - \frac{
\alpha }{\beta }
\biggr)^{2} 2 \lambda _{H}^{*} \rho
_{H}^{2},
\end{align*}
\begin{equation*}
\lambda _{H}^{*} = \frac{\lambda _{H}}{2 - 2H}, \qquad \rho
_{H} = \frac{\sqrt{\pi } \varGamma (3/2 - H)}{\gamma \kappa _{H}}.
\end{equation*}
The domain of the function $m_{1}^{(\alpha ,\beta )}$ equals $  \{  (
\xi _{1},\xi _{2})\in \mathbb{R}^{2}:D^{(\alpha ,\beta )}(\xi _{2})>0  \}
$.
\end{lemma}

The following lemma gives a joint MGF\index{joint MGF} of $  ( S_{T}, I_{T}, J_{T},
K_{T}   )$.
%
\begin{lemma}
\label{th:S_I_J_K_mgf}
The moment generating function of $  ( S_{T}, I_{T}, J_{T}, K_{T}
  )$ equals
\begin{align*}
m_{2}(\theta _{1}, \theta _{2}, \theta
_{3}, \theta _{4}) &= \mathbb{E} \bigl[ \exp \{ \theta
_{1} S_{T} + \theta _{2}
I_{T} + \theta _{3} J_{T} + \theta
_{4} K_{T} \} \bigr]
\\
&= m_{1}^{(\alpha _{1}, \beta _{1})} \biggl( \theta _{1} +
\frac{\alpha
-\alpha _{1}}{\gamma }, \theta _{2} - \beta + \beta _{1}
\biggr) \exp \biggl\{ \frac{\alpha _{1}^{2} - \alpha ^{2}}{2 \gamma ^{2}} w_{T}^{H}
\biggr\} ,
\end{align*}
where
\begin{equation*}
\alpha _{1} = - \frac{\gamma \theta _{3} + \alpha \beta }{\sqrt{\beta
^{2} - 2 \theta _{4}}}, \qquad \beta
_{1} = - \sqrt{\beta ^{2} - 2 \theta
_{4}}.
\end{equation*}
The domain of the function $m_{2}$ equals
\begin{equation*}
\left\{ (\theta _{1},\theta _{2},\theta
_{3},\theta _{4})\in \mathbb{R} ^{4} :
\theta _{4}<\beta ^{2}/2, D^{(\alpha ,\beta )} \left(\theta
_{2}- \beta -\sqrt{\beta ^{2} - 2 \theta
_{4}} \right)>0 \right\} ,
\end{equation*}
where $D^{(\alpha ,\beta )}$ is defined in \eqref{eq:D_def}.
\end{lemma}
\begin{proof}
The proof is the same as for
\cite[Theorem~3.4]{Lohvinenko-Ralchenko-2018}.
\end{proof}

\begin{lemma}
\label{l:S_asymptotic}
Under stated conditions the process $S_{T}$ has the normal asymptotic
distribution as $T \to \infty $, namely
%
\begin{equation}
\label{eq:S_asymptotic} T^{H - 1/2} e^{\beta T} S_{T}
\xrightarrow{d}\mathcal{N} \biggl( \frac{
  ( x_{0} - \frac{\alpha }{\beta }   ) \rho _{H} (-\beta )^{H
- 1/2}}{\sqrt{\pi }}, \frac{\varGamma (H) \varGamma (1 - H)}{2 \pi (-
\beta )\lambda _{H}^{*}}
\biggr).
\end{equation}
\end{lemma}
\begin{proof}
We obtain the distribution via MGF. Using Lemma~\ref{l:S_I_mgf} we have
\begin{equation*}
\mathbb{E} \bigl[ \exp \bigl\{ \theta T^{H - 1/2} e^{\beta T}
S_{T} \bigr\} \bigr] = m_{1} \bigl( \theta
T^{H - 1/2} e^{\beta T}, 0 \bigr).
\end{equation*}
Taking each term of the function $m_{1}$ separately with $\xi _{1} =
\theta T^{H - 1/2} e^{\beta T}$ and $\xi _{2} = 0$ and applying
\eqref{eq:I_asymptotic} we obtain that $D(\xi _{2}) = 1$, $A_{1}(\xi
_{1}, \xi _{2}) = A_{3}(\xi _{1}, \xi _{2})=0$,
\begin{align*}
A_{2}(\xi _{1}, \xi _{2}) &= \xi
_{1}^{2} c_{3} T^{2 - 2H}
e^{-\beta T} I _{1 - H} \biggl( -\frac{\beta T}{2} \biggr)
I_{H - 1} \biggl( -\frac{
\beta T}{2} \biggr)
\\
&= \theta ^{2} c_{3} T e^{\beta T}
I_{1 - H} \biggl( - \frac{\beta T}{2} \biggr) I_{H - 1}
\biggl( -\frac{\beta T}{2} \biggr)
\\
&\to \frac{c_{3}}{\pi (-\beta )} \theta ^{2}, \quad \text{as }T\to \infty
,
\end{align*}
{and}
\begin{align*}
A_{4}(\xi _{1}, \xi _{2}) &= \xi
_{1} 2 c _{1} \bigl(\tfrac{\alpha }{\beta } \bigr) (-
\beta )^{H} T^{1 - H} e^{-\frac{
\beta T}{2}} I_{1 - H}
\biggl( -\frac{\beta T}{2} \biggr)
\\
&= \theta 2 c_{1} \bigl(\tfrac{\alpha }{\beta } \bigr) (-\beta
)^{H} T ^{1/2} e^{\frac{\beta T}{2}} I_{1 - H}
\biggl( -\frac{\beta T}{2} \biggr)
\\
&\to \frac{2 c_{1}  (\tfrac{\alpha }{\beta }  ) (-\beta )^{H -
1/2}}{\sqrt{\pi }} \theta , \quad \text{as }T\to \infty .
\end{align*}
Hence
\begin{equation*}
\mathbb{E} \bigl[ \exp \bigl\{ \theta T^{H - 1/2} e^{\beta T}
S_{T} \bigr\} \bigr] \to \exp \biggl\{ \frac{c_{3}}{8 \pi (-\beta )} \theta
^{2} + \frac{c_{1}  (\tfrac{\alpha }{\beta }  ) (-\beta )^{H
- 1/2}}{4 \sqrt{\pi }} \theta \biggr\} ,
\end{equation*}
as $T \to \infty $. This means that
\begin{equation*}
T^{H - 1/2} e^{\beta T} S_{T} \xrightarrow{d}
\mathcal{N} \biggl( \frac{c
_{1}  (\tfrac{\alpha }{\beta }  ) (-\beta )^{H - 1/2}}{4 \sqrt{
\pi }}, \frac{c_{3}}{4 \pi (-\beta )} \biggr), \quad
\text{as }T \to \infty ,
\end{equation*}
which is equivalent to \eqref{eq:S_asymptotic}.
\end{proof}

The following result is crucial for the derivation of the joint
asymptotic distribution of MLE.

\begin{lemma}
\label{l:I+bK+K_mgf}
The vector of main components of the MLE has the following behavior
%
\begin{equation}
\label{eq:I+bK+K_mgf} %
\begin{pmatrix}
T^{H-1}  (S_{T}+\beta J_{T}-\frac{\alpha }{\gamma }w^{H}_{T}  )
\\
e^{\beta T} (I_{T} + \beta K_{T})
\\
e^{2 \beta T} K_{T}
\end{pmatrix} %
\xrightarrow{d}
\begin{pmatrix}
\xi
\\
\eta \zeta
\\
\zeta ^{2}
\end{pmatrix} %
, \quad \text{as } T\to \infty ,
\end{equation}
where $\xi $, $\eta $, $\zeta $ are independent and $\xi
\stackrel{d}{=}\mathcal{N}  (0, \lambda _{H}^{-1}  )$,
$\eta \stackrel{d}{=}\mathcal{N}(0, 1)$,
%
\begin{equation}
\label{eq:I+bK_asymptotic_zeta} \zeta \stackrel{d} {=}\mathcal{N} \biggl( \frac{  ( x_{0} - \frac{
\alpha }{\beta }   )\rho _{H}\sqrt{\lambda _{H}^{*} } (-\beta )^{H
- 1}}{\sqrt{2\pi }},
\frac{1}{4 \beta ^{2} \sin \pi H} \biggr).
\end{equation}
\end{lemma}
\begin{proof}
We again obtain the stated asymptotic distribution via MGF of the presented
vector. It could be easily reduced to already studied MGF. That said,
using Lemma~\ref{th:S_I_J_K_mgf}, we have
%
\begin{align}
&\mathbb{E} \biggl[ \exp \biggl\{ \theta _{1} T^{H-1}
\biggl(S_{T}+\beta J_{T}-\frac{\alpha }{\gamma }w^{H}_{T}
\biggr) + \theta _{2} e^{\beta T} (I_{T} + \beta
K_{T}) + \theta _{3} e^{2 \beta
T}
K_{T} \biggr\} \biggr]
\nonumber
\\*
&\quad = m_{2} \bigl(\theta _{1} T^{H-1},
\theta _{2} e^{\beta T}, \theta _{1} \beta
T^{H-1}, \theta _{2} \beta e^{\beta T} + \theta
_{3} e^{2
\beta T} \bigr) \exp \biggl\{ -\frac{\theta _{1} \alpha }{\gamma }T
^{H-1}w^{H}_{T} \biggr\}
\nonumber
\\
&\quad = m_{1}^{(\alpha _{1}(T), \beta _{1}(T))} \biggl(\theta _{1}
T^{H-1}+ \frac{
\alpha - \alpha _{1}(T)}{\gamma }, \theta _{2}
e^{\beta T} - \beta + \beta _{1}(T) \biggr)
\nonumber
\\
&\qquad \times \exp \biggl\{ \frac{\alpha _{1}(T)^{2} - \alpha ^{2}}{2
\gamma ^{2}} w_{T}^{H}
-\frac{\theta _{1} \alpha }{\gamma }T^{H-1}w^{H} _{T}
\biggr\} , \label{eq:mgf-3-stat}
\end{align}
where
\begin{align*}
\alpha _{1}(T) &= \frac{\beta \gamma \theta _{1}T^{H-1}+\alpha \beta }{-\sqrt{
\beta ^{2} - 2   ( \theta _{2} \beta e^{\beta T} + \theta _{3} e^{2
\beta T}   )}},
\\*
\beta _{1}(T) &= -\sqrt{\beta ^{2} - 2 \bigl( \theta
_{2} \beta e ^{\beta T} + \theta _{3}
e^{2 \beta T} \bigr)}.
\end{align*}
Applying the Taylor series expansion, we get as $T\to \infty $
%
\begin{align}
\alpha _{1}(T) &= \bigl(\gamma \theta _{1}T^{H-1}+
\alpha \bigr) \biggl[ 1 - 2 \biggl( \frac{\theta _{2}}{\beta } e^{\beta T} +
\frac{
\theta _{3}}{\beta ^{2}} e^{2 \beta T} \biggr) \biggr]^{-1/2}
\nonumber
\\
&= \bigl(\gamma \theta _{1}T^{H-1}+\alpha \bigr)
\biggl[ 1 + \biggl( \frac{\theta _{2}}{\beta } e^{\beta T} +
\frac{\theta _{3}}{\beta ^{2}} e ^{2 \beta T} \biggr) + O \bigl( e^{2 \beta T}
\bigr) \biggr]
\nonumber
\\
&= \alpha +\gamma \theta _{1}T^{H-1} +
\frac{\theta _{2}\alpha }{
\beta } e^{\beta T} + \frac{\theta _{2}\gamma \theta _{1}}{\beta } T
^{H-1} e^{\beta T} + O \bigl( e^{2 \beta T} \bigr)
\label{eq:alpha1-expand}
\end{align}
{and}
\begin{align}
\beta _{1}(T) &= \beta \biggl[ 1 - 2 \biggl( \theta
_{2} \frac{1}{\beta } e^{\beta T} + \theta
_{3} \frac{1}{\beta
^{2}} e^{2 \beta T} \biggr)
\biggr]^{1/2}
\nonumber
\\
&= \beta \biggl[ 1 - \biggl( \theta _{2} \frac{1}{\beta }
e^{\beta T} + \theta _{3} \frac{1}{\beta ^{2}}
e^{2 \beta T} \biggr)
\nonumber
\\*
& \qquad - \frac{1}{2} \biggl( \theta _{2}
\frac{1}{\beta } e^{\beta T} + \theta _{3}
\frac{1}{\beta ^{2}} e^{2 \beta T} \biggr)^{2} + O \bigl(
e^{3
\beta T} \bigr) \biggr]
\nonumber
\\
&= \beta - \theta _{2} e^{\beta T} + \frac{\theta _{2}^{2} + 2 \theta
_{3}}{2 (-\beta )}
e^{2 \beta T} + O \bigl( e^{3 \beta T} \bigr). \label{eq:beta1-expand}
\end{align}
Note that $\alpha _{1}(T)\to \alpha $ and $\beta _{1}(T)\to \beta $, as
$T\to \infty $. Moreover, the arguments of the function $m_{1}^{(
\alpha _{1}(T),\beta _{1}(T))}$ in \eqref{eq:mgf-3-stat} have the
following asymptotic behavior:
%
\begin{align}
\xi _{1}(T) &\coloneqq \theta _{1}T^{H-1}+
\frac{\alpha - \alpha _{1}(T)}{
\gamma } =\frac{\theta _{2}\alpha }{-\beta \gamma } e^{\beta T} +
\frac{
\theta _{1}\theta _{2}}{-\beta } T^{H-1} e^{\beta T}+ O \bigl(
e^{2
\beta T} \bigr), \label{eq:xi1-as}
\\
\xi _{2}(T) &\coloneqq \theta _{2} e^{\beta T} -
\beta + \beta _{1}(T) = \frac{\theta _{2}^{2} + 2 \theta _{3}}{2 (-\beta )} e^{2 \beta T} + O
\bigl( e^{3 \beta T} \bigr), \label{eq:xi2-as}
\end{align}
as $T\to \infty $. Further, inserting \eqref{eq:xi2-as} into
\eqref{eq:D_def}, and applying the expansion \eqref{eq:I_asymptotic}
from Appendix~\ref{sec:Bessel-function}, we obtain
%
\begin{align}
&D^{(\alpha _{1}(T), \beta _{1}(T))}\bigl(\xi _{2}(T)\bigr)
\nonumber\\
&\quad = \biggl(1 -
\frac{(\theta _{2}^{2} + 2 \theta _{3}) e^{2 \beta T}}{4 \beta _{1}(T)
(-\beta )} \biggr)^{2} + \frac{(\theta _{2}^{2} + 2 \theta _{3})^{2} e
^{4 \beta T}}{16 \beta _{1}(T)^{2} \beta ^{2}}
e^{-2 \beta _{1}(T) T}
\nonumber
\\
&\quad \quad + \biggl( \frac{(\theta _{2}^{2} + 2\theta _{3}) e^{2\beta T}}{2
\beta _{1}(T) (-\beta )} - \frac{(\theta _{2}^{2} + 2\theta _{3})^{2} e
^{4\beta T}}{8 \beta _{1}(T)^{2} \beta ^{2}} \biggr)
\frac{(-\beta _{1}(T)) \pi T}{4 \sin \pi H} e^{-\beta _{1}(T) T}
\nonumber
\\
& \quad \quad \times \biggl[ I_{-H} \biggl( -\frac{\beta _{1}(T) T}{2}
\biggr) I _{H - 1} \biggl( -\frac{\beta _{1}(T) T}{2} \biggr)
\nonumber
\\
&\quad  \quad  + I_{1 - H} \biggl( -\frac{\beta _{1}(T) T}{2} \biggr)
I_{H} \biggl( -\frac{
\beta _{1}(T) T}{2} \biggr) \biggr] + O \bigl(
e^{\beta T} \bigr)
\nonumber
\\
&\quad \sim \biggl(1 - \frac{(\theta _{2}^{2} + 2 \theta _{3}) e^{2 \beta T}}{4
\beta _{1}(T) (-\beta )} \biggr)^{2} +
\frac{(\theta _{2}^{2} + 2 \theta
_{3})^{2} e^{4 \beta T}}{16 \beta _{1}(T)^{2} \beta ^{2}} e^{-2 \beta
_{1}(T) T}
\nonumber
\\
&\quad \quad + \biggl( \frac{(\theta _{2}^{2} + 2\theta _{3}) e^{2\beta T}}{2
\beta _{1}(T) (-\beta )} - \frac{(\theta _{2}^{2} + 2\theta _{3})^{2} e
^{4\beta T}}{8 \beta _{1}(T)^{2} \beta ^{2}} \biggr)
\frac{1}{2 \sin
\pi H} e^{-2 \beta _{1}(T) T} + O \bigl( e^{\beta T} \bigr)
\nonumber
\\
&\quad \to 1 - \frac{\theta _{2}^{2} + 2 \theta _{3}}{4 \beta ^{2} \sin
\pi H}, \quad \text{as } T\to \infty . \label{eq:D'}
\end{align}
It follows from \eqref{eq:xi1-as}, \eqref{eq:xi2-as} that
%
\begin{equation}
\label{eq:c1-c2} c_{1} \bigl(\tfrac{\alpha _{1}(T)}{\beta _{1}(T)} \bigr) \xi
_{1}(T) - c _{2} \bigl(\tfrac{\alpha _{1}(T)}{\beta _{1}(T)} \bigr) \xi
_{2}(T) \sim c_{1} \bigl(\tfrac{\alpha }{\beta } \bigr)
\frac{\theta _{2}\alpha }{-
\beta \gamma } e^{\beta T}, \quad \text{as }T\to \infty .
\end{equation}
Using this relation, \eqref{eq:xi2-as} and \eqref{eq:I_asymptotic}, we
get from \eqref{eq:A_1_def} that
%
\begin{align}
&A_{1}^{(\alpha _{1}(T), \beta _{1}(T))}\bigl(\xi _{1}(T), \xi
_{2}(T)\bigr)\nonumber\\
&\quad  = \xi _{2}(T) \bigl(c_{1}
\bigl(\tfrac{\alpha _{1}(T)}{
\beta _{1}(T)} \bigr) \xi _{1}(T) - c_{2}
\bigl(\tfrac{\alpha _{1}(T)}{
\beta _{1}(T)} \bigr)\xi _{2}(T) \bigr)
\nonumber
\\*
& \qquad \times \bigl(-\beta _{1}(T)\bigr)^{H-1}
T^{1 - H} e^{-\frac{3 \beta _{1}(T) T}{2}} I_{1 - H} \biggl( -
\frac{\beta _{1}(T) T}{2} \biggr)
\nonumber
\\
&\quad \sim \frac{\theta _{2}^{2} + 2 \theta _{3}}{2 (-\beta )} e^{2 \beta T} c_{1} \bigl(
\tfrac{\alpha }{\beta } \bigr) \frac{\alpha }{(-\beta )
\gamma } \theta _{2}
e^{\beta T} \bigl(-\beta _{1}(T)\bigr)^{H-1}
T^{1 - H}
\nonumber
\\*
& \qquad \times e^{-\frac{3 \beta _{1}(T) T}{2}} I_{1 - H} \biggl( -
\frac{\beta
_{1}(T) T}{2} \biggr)
\nonumber
\\
&\quad \sim \frac{\alpha \theta _{2}  (\theta _{2}^{2} + 2
\theta _{3}  )}{2 \beta ^{2}\gamma } e^{3 \beta T} c_{1} \bigl(
\tfrac{
\alpha }{\beta } \bigr) \bigl(-\beta _{1}(T)
\bigr)^{H-3/2} \frac{1}{\sqrt{
\pi }} T^{1/2 - H}
e^{-2 \beta _{1}(T) T}
\nonumber
\\
&\quad = O \bigl( T^{1/2 - H} e^{\beta T} \bigr) \to 0, \quad
\text{as }T \to \infty . \label{eq:A1'}
\end{align}
It follows from \eqref{eq:xi1-as}, \eqref{eq:xi2-as} that
\begin{gather*}
\xi _{1}(T)^{2} = \frac{\theta _{2}^{2}\alpha ^{2}}{\beta ^{2}\gamma ^{2}}
e^{2\beta T} + O \bigl( T^{H-1}e^{2 \beta T} \bigr),
\\
\xi _{1}(T)\xi _{2}(T) = O \bigl( e^{3 \beta T}
\bigr), \qquad \xi _{2}(T)^{2} = O \bigl(e^{4 \beta T}
\bigr),
\end{gather*}
as $T\to \infty $. Therefore, by \eqref{eq:A_2_def} and
\eqref{eq:I_asymptotic} we obtain
%
\begin{align}
&A_{2}^{(\alpha _{1}(T), \beta _{1}(T))}\bigl(\xi _{1}(T), \xi
_{2}(T)\bigr)
\nonumber
\\*
&\quad = \bigl( \xi _{1}(T)^{2} c_{3} -
\xi _{1}(T) \xi _{2}(T) c_{4} \bigl(
\tfrac{
\alpha _{1}(T)}{\beta _{1}(T)} \bigr) + \xi _{2}(T)^{2}
c_{5} \bigl(\tfrac{
\alpha _{1}(T)}{\beta _{1}(T)} \bigr) \bigr)
\nonumber
\\*
& \qquad \times T^{2 - 2H} e^{-\beta _{1}(T) T} I_{1 - H} \biggl(
-\frac{\beta
_{1}(T) T}{2} \biggr) I_{H - 1} \biggl( -\frac{\beta _{1}(T) T}{2}
\biggr)
\nonumber
\\
&\quad \sim c_{3} \frac{\alpha ^{2}}{\beta ^{2} \gamma ^{2}} \theta _{2}^{2}
e ^{2 \beta T} T^{2 - 2H} e^{- \beta _{1}(T) T} I_{1 - H}
\biggl( -\frac{
\beta _{1}(T) T}{2} \biggr) I_{H - 1} \biggl( -
\frac{\beta _{1}(T) T}{2} \biggr)
\nonumber
\\
&\quad \sim e^{2 \beta T} c_{3} \frac{\alpha ^{2}}{\beta ^{2} \gamma ^{2}} \theta
_{2}^{2} \frac{1}{(-\beta ) \pi } T^{1 - 2H}
e^{- 2 \beta _{1}(T)
T}
\nonumber
\\*
&\quad =O \bigl(T^{1 - 2H} \bigr)\to 0, \quad \text{as }T\to \infty .
\label{eq:A2'}
\end{align}

Note that $\xi _{2}(T)-2\beta _{1}(T)\to -2\beta $, as $T\to \infty $, by
\eqref{eq:beta1-expand} and \eqref{eq:xi2-as}. Hence, by
\eqref{eq:A_3_def} and \eqref{eq:I_asymptotic},\vadjust{\goodbreak}
%
\begin{align}
& A_{3}^{(\alpha _{1}(T), \beta _{1}(T))}\bigl(\xi _{1}(T), \xi
_{2}(T)\bigr)
\nonumber\\
&\quad = \xi _{2}(T) \bigl( \xi
_{2}(T) - 2 \beta _{1}(T) \bigr) c_{6}
\bigl(\tfrac{\alpha _{1}(T)}{\beta _{1}(T)} \bigr)
\nonumber
\\*
& \qquad \times \bigl(-\beta _{1}(T)\bigr)^{2H-1} T
e^{-\beta _{1}(T) T} I_{1 - H} \biggl( -\frac{\beta _{1}(T) T}{2} \biggr)
I_{-H} \biggl( -\frac{\beta _{1}(T)
T}{2} \biggr)
\nonumber
\\
&\quad \sim \frac{\theta _{2}^{2} + 2 \theta _{3}}{-2\beta } e^{2 \beta T}(-2 \beta ) c_{6}
\bigl(\tfrac{\alpha }{\beta } \bigr) \bigl(-\beta _{1}(T)
\bigr)^{2H-2} \frac{1}{\pi } e^{- 2 \beta _{1}(T) T}
\nonumber
\\
&\quad \to \frac{c_{6}  (\frac{\alpha }{\beta }  ) (-\beta )^{2H-2}}{
\pi } \bigl(\theta _{2}^{2} +2
\theta _{3}\bigr), \quad \text{as }T\to \infty . \label{eq:A3'}
\end{align}
Similarly, using \eqref{eq:A_4_def}, \eqref{eq:c1-c2} and
\eqref{eq:I_asymptotic}, we get
%
\begin{align}
& A_{4}^{(\alpha _{1}(T), \beta _{1}(T))}\bigl(\xi _{1}(T), \xi
_{2}(T)\bigr)
\nonumber\\
&\quad = \bigl(c_{1} \bigl(\tfrac{\alpha _{1}(T)}{\beta _{1}(T)}
\bigr) \xi _{1}(T) - c_{2} \bigl(\tfrac{\alpha _{1}(T)}{\beta _{1}(T)}
\bigr) \xi _{2}(T) \bigr)
\nonumber
\\*
& \qquad \times \bigl(\xi _{2}(T) - 2 \beta _{1}(T)
\bigr) \bigl(-\beta _{1}(T)\bigr)^{H - 1}
T^{1 - H} e^{-\frac{\beta _{1}(T) T}{2}} I_{1 - H} \biggl( -
\frac{\beta _{1}(T) T}{2} \biggr)
\nonumber
\\
&\quad \sim c_{1} \bigl(\tfrac{\alpha }{\beta } \bigr)
\frac{\theta _{2}\alpha
}{-\beta \gamma } e^{\beta T} (- 2 \beta ) \bigl(-\beta
_{1}(T)\bigr)^{H - 3/2} \frac{1}{\sqrt{
\pi }}
T^{1/2 - H} e^{- \beta _{1}(T) T}
\nonumber
\\
&\quad =O \bigl( T^{1/2-H} \bigr) \to 0, \quad \text{as }T\to \infty .
\label{eq:A4'}
\end{align}
Also, \eqref{eq:alpha1-expand} implies
\begin{equation*}
\alpha _{1}(T)^{2} = \alpha ^{2} +2\alpha
\gamma \theta _{1}T^{H-1} + \gamma ^{2}\theta
_{1}^{2}T^{2H-2} + O \bigl(e^{\beta T}
\bigr), \quad \text{as }T\to \infty .
\end{equation*}
Then, for the expression under the exponential function in
\eqref{eq:mgf-3-stat} we have
%
\begin{align}
& \frac{\alpha _{1}(T)^{2} - \alpha ^{2}}{2 \gamma ^{2}} w_{T} ^{H} -
\frac{\theta _{1} \alpha }{\gamma }T^{H-1}w^{H}_{T} =
\frac{1}{2} \theta _{1}^{2}T^{2H-2}w_{T}^{H}
+ O \bigl(w_{T}^{H}e^{\beta T}\bigr)
\nonumber
\\
&\quad =\frac{1}{2}\theta _{1}^{2}\lambda
_{H}^{-1} + O \bigl(w_{T}^{H}e^{\beta T}
\bigr) \to \frac{1}{2}\theta _{1}^{2}\lambda
_{H}^{-1}, \quad \text{as }T \to \infty ,
\label{eq:exponent}
\end{align}
since $w^{H}_{T}=\lambda _{H}^{-1}T^{2-2H}$.

Thus, inserting \eqref{eq:D'} and \eqref{eq:A1'}--\eqref{eq:exponent}
into \eqref{eq:mgf-3-stat}, we arrive at
\begin{align*}
& \mathbb{E} \biggl[ \exp \biggl\{ \theta _{1} T^{H-1}
\biggl(S_{T}+\beta J_{T}-\frac{\alpha }{\gamma }w^{H}_{T}
\biggr) + \theta _{2} e^{\beta T} (I_{T} + \beta
K_{T}) + \theta _{3} e^{2 \beta
T}
K_{T} \biggr\} \biggr]
\\*
&\quad \to \exp \biggl\{ \frac{1}{2}\theta _{1}^{2}
\lambda _{H}^{-1} \biggr\} \biggl( 1 -
\frac{\theta _{2}^{2} + 2 \theta _{3}}{4 \beta ^{2} \sin
\pi H} \biggr)^{-1/2}
\\
&\qquad \times \exp \biggl\{ \frac{c_{6}(\frac{\alpha }{\beta }) (-
\beta )^{2H-2}   ( \theta _{2}^{2} + 2 \theta _{3}   )}{8
\pi   \Bigl( 1 - \frac{\theta _{2}^{2} + 2 \theta _{3}}{4 \beta ^{2}
\sin \pi H}   \Bigr)} \biggr\} ,\quad \text{as } T
\to \infty .
\end{align*}
We see that the limit is a product of MGF of the normal random variable
$\xi \stackrel{d}{=}\mathcal{N}  (0,\lambda _{H}^{-1}  )$ and
MGF of the random vector $\binom{\eta \zeta }{\zeta ^{2}}$, where the random variables
$\eta \stackrel{d}{=}\mathcal{N}(0, 1)$ and $ \zeta \stackrel{d}{=}
\mathcal{N}  \biggl( \frac{\sqrt{c_{6}(\frac{\alpha }{\beta })}(-
\beta )^{H - 1}}{2\sqrt{\pi }},\allowbreak \frac{1}{4 \beta ^{2} \sin \pi H}
  \biggr)$ are independent, see Lemma~\ref{l:product_mgf_for_I+bK} in
Appendix~\ref{sec:product_mgf_general}. This concludes the proof, since
$c_{6}(\frac{\alpha }{\beta }) = ( x_{0} - \frac{\alpha }{\beta })^{2}
2 \lambda _{H}^{*} \rho _{H}^{2}$.
\end{proof}

\begin{remark}
\label{rem:normality}
In fact, $\mathcal{N}  (0, \lambda _{H}^{-1}  )$ is an exact
distribution of the random variable\break $T^{H-1}  (S_{T}+\beta J_{T}-\frac{
\alpha }{\gamma }w^{H}_{T}  )$ for any $T$. It can be easily seen
from the above proof by putting $\theta _{2}=\theta _{3}=0$ (then
$\alpha _{1}(T)=\alpha + \gamma \theta _{1}T^{H-1}$, $\beta _{1}(T)=
\beta $, $\xi _{1}(T)=\xi _{2}(T)=0$).
\end{remark}

The following series of corollaries will describe asymptotic
distributions of minor components of the MLE.

First, by considering the convergence of the first component of the
random vector in \eqref{eq:I+bK+K_mgf}, we immediately get the following
result.
%
\begin{cor}
\label{l:S+bJ_asymptotic}
For the process $(S_{T} + \beta J_{T})$ we have
\begin{equation*}
\frac{1}{w_{T}^{H}} (S_{T} + \beta J_{T})
\xrightarrow{\mathsf{P}} \frac{
\alpha }{\gamma }, \quad T \to \infty .
\end{equation*}
\end{cor}

Next, we focus on the process $I_{T}$. In order to obtain its
asymptotic behavior it suffices to write
\begin{equation*}
e^{2 \beta T} I_{T} = e^{2 \beta T} (I_{T}+
\beta K_{T}) - \beta e^{2
\beta T} K_{T}
\end{equation*}
and then apply \eqref{eq:I+bK+K_mgf}.
%
\begin{cor}
\label{l:I_T_asymptotic}
For the process $I_{T}$ we have
\begin{equation*}
e^{2 \beta T} I_{T} \xrightarrow{d}-\beta \zeta
^{2}, \quad T \to \infty ,
\end{equation*}
where $\zeta $ has the normal distribution defined in
\eqref{eq:I+bK_asymptotic_zeta}.
\end{cor}
Finally, the asymptotic behavior of $J_{T}$ can be easily derived using
Lemma~\ref{l:S_asymptotic}, Corollary \ref{l:S+bJ_asymptotic} and the
identity
\begin{equation*}
-\beta T^{H - \frac{1}{2}} e^{\beta T} J_{T} =
T^{H - \frac{1}{2}} e ^{\beta T} S_{T} - T^{H - \frac{1}{2}}
e^{\beta T} (S_{T} + \beta J_{T} ).
\end{equation*}
%
\begin{cor}
\label{l:J_asymptotic}
For the process $J_{T}$ we have
\begin{equation*}
T^{H - \frac{1}{2}} e^{\beta T} J_{T} \xrightarrow{d}
\mathcal{N} \biggl( \frac{8 ( x_{0} - \frac{\alpha }{\beta })\rho _{H} (-\beta )^{H
- 3/2}}{\sqrt{\pi }}, \frac{4 \varGamma (H) \varGamma (1 - H)}{\lambda
_{H}^{*}(-\beta )^{3} \pi } \biggr),\quad  T
\to \infty .
\end{equation*}
\end{cor}

\section{Main results}%
\label{sec:main}
Now we are ready to prove the main result of the article.
%
\begin{thm}
\label{th:alpha_asymptotic}
Let $\beta < 0$, $H \in (1/2,1)$. Then
%
\begin{equation}
\label{eq:alpha_asymptotic} %
\begin{pmatrix}
T^{1-H}   ( \widehat{\alpha }_{T} - \alpha   )
\\
e^{-\beta T}   ( \widehat{\beta }_{T} - \beta   )
\end{pmatrix} %
\xrightarrow{d}
\begin{pmatrix}
\nu
\\
\frac{\eta }{\zeta }
\end{pmatrix} %
, \quad T \to \infty ,
\end{equation}
where $\nu \stackrel{d}{=}\mathcal{N}(0, \lambda _{H} \gamma ^{2})$,
$\eta \stackrel{d}{=}\mathcal{N}(0, 1)$, and
%
\begin{equation}
\label{eq:alpha_asymptotic_zeta} \zeta \stackrel{d} {=}\mathcal{N} \biggl( \frac{( x_{0} - \frac{\alpha
}{\beta })\rho _{H}\sqrt{\lambda _{H}^{*} } (-\beta )^{H - 1}}{\sqrt{2
\pi }},
\frac{1}{4 \beta ^{2} \sin \pi H} \biggr)
\end{equation}
are independent random variables.\index{random variables} In particular, the estimators
$\widehat{\alpha }_{T}$ and $\widehat{\beta }_{T}$ are asymptotically
independent.
\end{thm}

\begin{proof}
Using \eqref{eq:mle} and the equality $T^{1-H}=\lambda _{H} w^{H}_{T} T
^{H-1}$, we can write
%
\begin{align}
T^{1-H} ( \widehat{\alpha }_{T} - \alpha )& = \lambda
_{H} w^{H}_{T} T^{H-1} \biggl(
\frac{S_{T} K_{T} - I_{T} J_{T}}{w_{T}
^{H} K_{T} - J_{T}^{2}}  \gamma - \alpha \biggr)
\nonumber
\\
&= \lambda _{H} w^{H}_{T}
T^{H-1} \frac{\gamma S_{T} K_{T} - \gamma I
_{T} J_{T} - \alpha w_{T}^{H} K_{T} + \alpha J_{T}^{2}}{w_{T}^{H} K
_{T} - J_{T}^{2}}
\nonumber
\\
&= \frac{e^{2 \beta T} K_{T} \gamma \lambda _{H} T^{H-1}(S_{T} +
\beta J_{T} - \frac{\alpha }{\gamma } w_{T}^{H})}{e^{2 \beta T} K_{T}
- \frac{1}{w_{T}^{H}} e^{2 \beta T} J_{T}^{2}}
\nonumber
\\
&\quad +\frac{- \gamma \lambda _{H} T^{H-1} e^{2\beta T} J_{T}(I_{T} +
\beta K_{T}) + \alpha \lambda _{H} T^{H-1} e^{2 \beta T} J_{T}^{2}}{e
^{2 \beta T} K_{T} - \frac{1}{w_{T}^{H}} e^{2 \beta T} J_{T}^{2}}. \label{eq:alpha-asymptotic}
\end{align}
Note that by Corollary~\ref{l:J_asymptotic} and
Lemma~\ref{l:I+bK+K_mgf}, we see that $J_{T} = O_{\mathsf{P}}(T^{
\frac{1}{2}-H} e^{-\beta T})$,
$I_{T} + \beta K_{T} = O_{\mathsf{P}}(e^{-\beta T})$, and $e^{2
\beta T} K_{T} - \frac{1}{w_{T}^{H}} e^{2 \beta T} J_{T}^{2}
\xrightarrow{d}\zeta ^{2}$, as $T\to \infty $. Consequently, the second
term in the right-hand side of \eqref{eq:alpha-asymptotic} converges to
zero in probability.

Further, by \eqref{eq:mle},
%
\begin{align}
e^{-\beta T} ( \widehat{\beta }_{T} - \beta ) &=
e^{-
\beta T} \biggl( \frac{S_{T} J_{T} - w_{T}^{H} I_{T}}{w_{T}^{H} K_{T}
- J_{T}^{2}} - \beta \biggr)
\nonumber
\\
&= e^{-\beta T} \frac{S_{T} J_{T} - w_{T}^{H} I_{T} - \beta w_{T}
^{H} K_{T} + \beta J_{T}^{2}}{w_{T}^{H} K_{T} - J_{T}^{2}}
\nonumber
\\
&= \frac{- e^{\beta T}(I_{T} + \beta K_{T}) + e^{\beta T} J_{T} \frac{1}{w
_{T}^{H}}(S_{T} + \beta J_{T})}{e^{2 \beta T} K_{T} - \frac{1}{w_{T}
^{H}} e^{2 \beta T} J_{T}^{2}}. \label{eq:beta-asymptotic}
\end{align}
Corollaries~\ref{l:S+bJ_asymptotic} and \ref{l:J_asymptotic} imply that
$e^{\beta T} J_{T} \frac{1}{w_{T}^{H}}(S_{T} + \beta J_{T})$ converges
to zero in probability. Then applying Lemma~\ref{l:I+bK+K_mgf} and
Slutsky's theorem, from \eqref{eq:alpha-asymptotic}, \eqref{eq:beta-asymptotic} we get the convergence \eqref{eq:alpha_asymptotic}.
\end{proof}

\begin{remark}
Unlike the ergodic case (studied in
\cite{Lohvinenko-Ralchenko-2018}), in the non-ergodic case the initial
value $x_{0}$ affects the asymptotic bias of $\widehat{\beta }_{T}$. If
$\beta <0$, then the deterministic term $ (x_{0} -\frac{\alpha }{
\beta }) e^{-\beta t}$\index{deterministic term} in \eqref{eq:solution} does not converge to zero
and, moreover, has the same asymptotic order $O(e^{-\beta T})$ as the
stochastic term $\gamma \int_{0}^{t} e^{-\beta (t-s)} dB_{s}^{H}$. This
implies that the asymptotic behavior of the statistics $S_{T}$,
$I_{T}$, $J_{T}$, and $K_{T}$ depends on $x_{0}$. A similar dependence on
initial condition holds for the non-ergodic Ornstein--Uhlenbeck model
driven by the Brownian motion (see \cite{Feigin76} and
\cite[Prop.~3.46]{Kutoyants04}) and some explosive autoregressive models\index{explosive autoregressive models}
\cite{Anderson59,WangYu15,White58}.
\end{remark}

\begin{remark}
\label{rem:part-case}
The model \eqref{eq:fr-vas} with $x_{0}=\frac{\alpha }{\beta }$ was
considered in \cite[Th.~5.2]{Tanaka-Xiao-Yu}. In this particular
case we have $\zeta \stackrel{d}{=}\mathcal{N}  ( 0, \frac{1}{4
\beta ^{2} \sin \pi H}   )$. Consequently,\vadjust{\goodbreak}
\begin{equation*}
\frac{e^{-\beta T}}{2\beta } ( \widehat{\beta }_{T} - \beta )
\xrightarrow{d}\frac{X\sqrt{\sin \pi H}}{Y}, \quad \text{as }T\to \infty ,
\end{equation*}
where $X$ and $Y$ are two independent $\mathcal{N}(0,1)$ random
variables.\index{random variables} This completely agrees with
\cite[Th.~5.2]{Tanaka-Xiao-Yu}.
\end{remark}

\begin{remark}[Alternative parameterization]
An alternative specification of the fractional Vasicek model\index{fractional Vasicek model} is
%
\begin{equation}
\label{eq:alt_params_process_def} dX_{t} = \kappa ( \mu - X_{t} ) dt +
\gamma dB_{t}^{H}, \quad t\in [0,T], \quad
X_{0}=x_{0}.
\end{equation}
For the model \eqref{eq:alt_params_process_def}, the MLEs\index{maximum likelihood estimator (MLE)} of the
parameters $\mu $ and $\kappa $ have the following form
\cite{Lohvinenko-Ralchenko-2018}:
\begin{equation*}
\widehat{\mu }_{T} = \frac{ S_{T} K_{T} - I_{T} J_{T} }{ S_{T} J_{T} -
w_{T}^{H} I_{T}}  \gamma , \qquad
\widehat{\kappa }_{T} = \frac{S_{T} J_{T} - w_{T}^{H} I_{T}}{w_{T}
^{H} K_{T} - J_{T}^{2}}.
\end{equation*}
One can establish the following result: if $\kappa < 0$ and
$H \in (1/2,1)$, then
%
\begin{equation}
\label{eq:mu_asymptotic} %
\begin{pmatrix}
T^{1-H}   ( \widehat{\mu }_{T} - \mu   )
\\
e^{-\kappa T}   ( \widehat{\kappa }_{T} - \kappa   )
\end{pmatrix} %
\xrightarrow{d}
\begin{pmatrix}
\widetilde{\nu }
\\
\eta /\widetilde{\zeta }
\end{pmatrix} %
, \quad \text{as } T \to \infty ,
\end{equation}
where $\widetilde{\nu } \stackrel{d}{=}\mathcal{N}  (0, \frac{
\lambda _{H} \gamma ^{2}}{\kappa ^{2}}  )$, $\eta \stackrel{d}{=}
\mathcal{N}(0, 1)$, and
\begin{equation*}
\widetilde{\zeta } \stackrel{d} {=}\mathcal{N} \biggl( \frac{  ( x
_{0} - \mu   )\rho _{H}\sqrt{\lambda _{H}^{*} } (-\kappa )^{H - 1}}{\sqrt{2
\pi }},
\frac{1}{4 \kappa ^{2} \sin \pi H} \biggr),
\end{equation*}
are independent random variables.\index{random variables}

The proof of \eqref{eq:mu_asymptotic} is carried out by the delta-method.
By Taylor's theorem, for the function $g(x,y)=\frac{x}{y}$, we have as
$(x,y)\to (\alpha ,\beta )$
%
\begin{equation}
g(x,y)-g(\alpha ,\beta ) = \frac{1}{\beta }(x-\alpha ) -
\frac{\alpha
}{\beta ^{2}}(y-\beta ) + o \Bigl(\sqrt{(x-\alpha )^{2}+(y-
\beta )^{2}} \,\Bigr). \label{eq:taylor}
\end{equation}
Multiplying both sides of \eqref{eq:taylor} by $T^{1-H}$, and putting
$x=\widehat{\alpha }_{T}$, $y=\widehat{\beta }_{T}$, we get
%
\begin{equation}
\label{eq:taylor2} T^{1-H} \biggl(\frac{\widehat{\alpha }_{T}}{\widehat{\beta }_{T}} -
\frac{
\alpha }{\beta } \biggr) = \frac{1}{\beta }T^{1-H} (\widehat{
\alpha } _{T}-\alpha ) + R_{T},
\end{equation}
where
%
\begin{equation}
\label{eq:R->0} R_{T} = - \frac{\alpha }{\beta ^{2}}T^{1-H}
(\widehat{\beta }_{T}- \beta ) + o_{\mathsf{P}}
\Bigl(T^{1-H}\sqrt{ (\widehat{\alpha } _{T}-\alpha
)^{2}+ (\widehat{\beta }_{T}-\beta ) ^{2}}
\,\Bigr) \topr 0,
\end{equation}
as $T\to \infty $, since
$T^{1-H}  (\widehat{\beta }_{T}-\beta   )\topr 0$ and
$T^{1-H}  (\widehat{\alpha }_{T}-\alpha   )\xrightarrow{d}
\nu $ due to \eqref{eq:alpha_asymptotic}.

Finally, by Slutsky's theorem, from \eqref{eq:alpha_asymptotic},
\eqref{eq:taylor2} and \eqref{eq:R->0} we obtain the convergence
\begin{equation*}
\begin{pmatrix}
T^{1-H}  (\frac{\widehat{\alpha }_{T}}{\widehat{\beta }_{T}} - \frac{
\alpha }{\beta }  )
\\
e^{-\beta T}   (\widehat{\beta }_{T} - \beta   )
\end{pmatrix} %
= %
\begin{pmatrix}
\frac{1}{\beta }T^{1-H}   ( \widehat{\alpha }_{T} - \alpha   )
\\
e^{-\beta T}   ( \widehat{\beta }_{T} - \beta   )
\end{pmatrix} %
+ %
\begin{pmatrix}
R_{T}
\\
0
\end{pmatrix}
\xrightarrow{d} %
\begin{pmatrix}
\frac{\nu }{\beta }
\\
\frac{\eta }{\zeta }
\end{pmatrix} %
, \quad
\text{as } T\to \infty ,
\end{equation*}
which is equivalent to \eqref{eq:mu_asymptotic}, since $
\widehat{\mu }_{T} = \frac{\widehat{\alpha }_{T}}{\widehat{\beta }
_{T}}$, $\widehat{\kappa }_{T} = \widehat{\beta }_{T}$, $\mu = \frac{
\alpha }{\beta }$, $\kappa = \beta $.
\end{remark}

Now let us consider the situation when one of the parameters is known.
In this case we can obtain strong consistency of the corresponding MLEs
(instead of weak one) by applying the strong law of large numbers for
martingales,\index{martingales} see, e.g.,
\cite[Theorem~2.6.10]{liptser-shiryaev1}.

\begin{thm}
\label{th:MLE_alpha}
Let $\beta <0$ be known and $H \in (1/2,1)$. The MLE for $\alpha $ is
%
\begin{equation}
\label{eq:al-mle} \widetilde{\alpha }_{T} = \frac{\gamma }{ w_{T}^{H} }
(S_{T} + \beta J_{T} ).
\end{equation}
It is unbiased, strongly consistent and normal:
%
\begin{equation}
\label{eq:al-norm} T^{1 - H} ( \widetilde{\alpha }_{T} -
\alpha ) \stackrel{d} {=}\mathcal{N} \bigl(0, \lambda _{H} \gamma
^{2} \bigr).
\end{equation}
\end{thm}

\begin{proof}
The form of the MLE \eqref{eq:al-mle} was established in
\cite[Eq.~(3.2)]{Lohvinenko-Ralchenko-2017}. The normality follows from
Remark~\ref{rem:normality}:
\begin{equation*}
T^{1-H} (\widetilde{\alpha }_{T} - \alpha ) = \lambda
_{H} \gamma T^{H-1} \biggl(S_{T}+\beta
J_{T}-\frac{\alpha }{\gamma }w^{H} _{T}
\biggr) \stackrel{d} {=}\mathcal{N} \bigl(0, \lambda _{H} \gamma
^{2} \bigr).
\end{equation*}
In order to obtain the strong consistency, we rewrite this equality
using the relation
%
\begin{equation}
\label{eq:M^H} S_{T}+\beta J_{T}-
\frac{\alpha }{\gamma }w^{H}_{T} = M^{H}_{T},
\end{equation}
which follows from \eqref{eq:SviaM} and \eqref{eq:Q_via_P}. Then
\begin{equation*}
\widetilde{\alpha }_{T} - \alpha = \gamma \frac{M^{H}_{T}}{w^{H}_{T}}.
\end{equation*}
Recall that $M^{H}_{T}$ is a martingale\index{martingales} and $\langle M^{H}\rangle _{T}
= w^{H}_{T}\to \infty $, as $T\to \infty $. Then $\frac{M^{H}_{T}}{w
^{H}_{T}} \to 0$ a.s., as $T\to \infty $, by the strong law of large
numbers for martingales\index{martingales}
\cite[Th.~2.6.10, Cor.~1]{liptser-shiryaev1}.
\end{proof}

\begin{remark}
Actually, the statement of Theorem~\ref{th:MLE_alpha} is true regardless
of the sign of $\beta $. For $\beta >0$, \eqref{eq:al-norm} was proved
in \cite[Th.~3.1]{Lohvinenko-Ralchenko-2017}.
\end{remark}

\begin{thm}
\label{th:MLE_beta}
Let $\alpha $ be known, $H \in (1/2,1)$ and $\beta <0$. The MLE for
$\beta $ is
%
\begin{equation}
\label{eq:be-mle} \widetilde{\beta }_{T} = \frac{ \frac{\alpha }{\gamma } J_{T} - I_{T}
}{ K_{T}}.
\end{equation}
It is strongly consistent and
%
\begin{equation}
\label{eq:beta-asymp} e^{-\beta T} ( \widetilde{\beta }_{T} -
\beta ) \xrightarrow{d}\frac{\eta }{\zeta }, \quad \text{as }T\to \infty ,
\end{equation}
where $\eta $ and $\zeta $ are the same as in
Theorem~\ref{th:alpha_asymptotic}.
\end{thm}

\begin{proof}
The form of the MLE \eqref{eq:be-mle} is found in
\cite[Eq.~(3.3)]{Lohvinenko-Ralchenko-2017}. The strong consistency is
established in the same way as in
\cite[Th.~3.2]{Lohvinenko-Ralchenko-2017}. It follows from
\eqref{eq:M^H} that\vadjust{\goodbreak}
\begin{align*}
\frac{\alpha }{\gamma } J_{T} - I_{T} - \beta
K_{T} &=\frac{\alpha }{
\gamma }\int_{0}^{T}
P_{H}(t)\,dw^{H}_{T}-\int
_{0}^{T} P_{H}(t)\,dS
_{t}-\beta \int_{0}^{T}
\bigl(P_{H}(t) \bigr)^{2}\,dw^{H}_{t}
\\
& = - \int_{0}^{T} P_{H}(t)
\,dM^{H}_{T}.
\end{align*}
Whence, \eqref{eq:be-mle} implies that
\begin{equation*}
\widetilde{\beta }_{T} - \beta = \frac{\frac{\alpha }{\gamma } J_{T} -
I_{T} - \beta K_{T}}{K_{T}} = -
\frac{\int_{0}^{T} P_{H}(t)\,dM^{H}
_{T}}{\int_{0}^{T}   ( P_{H}(t)   )^{2} \, d w_{t}^{H}}.
\end{equation*}
Since the process $M^{H}$ is a martingale\index{martingales} with the quadratic variation
$w^{H}$, the process\break $\int_{0}^{\cdot }P_{H}(t) \, d M_{t}^{H}$ is a
martingale\index{martingales} with the quadratic variation $\int_{0}^{\cdot }  ( P_{H}(t)
  )^{2} \, d w_{t}^{H}$. Note that $\int_{0}^{T}   ( P_{H}(t)
  )^{2} \, d w_{t}^{H} = K_{T} \to \infty $ in probablility,\index{probablility} by
Lemma \ref{l:I+bK+K_mgf}. This convergence holds almost surely, because
$\int_{0}^{T}   ( P_{H}(t)   )^{2} \, d w_{t}^{H}$ is
increasing in upper bound $T$. Therefore, by the strong law of large
numbers for martingales\index{martingales}~\cite[Th.~2.6.10,
Cor.~1]{liptser-shiryaev1}, we get that $\widetilde{\beta }_{T}
\to \beta $ a.s., as $T \to \infty $.

Finally, the convergence \eqref{eq:beta-asymp} follows from the
representation
\begin{equation*}
e^{-\beta T} ( \widetilde{\beta }_{T} - \beta ) =
\frac{\frac{
\alpha }{\gamma } e^{\beta T} J_{T} - e^{\beta T}(I_{T} + \beta K_{T})}{e
^{2\beta T} K_{T}},
\end{equation*}
Lemma \ref{l:I+bK+K_mgf}, Corollary \ref{l:J_asymptotic}, and Slutsky's
theorem.
\end{proof}

\begin{remark}
The particular case when the parameter $\alpha =0$ is known and
$x_{0}=0$ was studied in \cite{Tanaka2015}. Similarly to
Remark~\ref{rem:part-case}, we see that in this case the
convergence~\eqref{eq:beta-asymp} takes the form
\begin{equation*}
\frac{e^{-\beta T}}{2\beta } ( \widetilde{\beta }_{T} - \beta )
\xrightarrow{d}\frac{X\sqrt{\sin \pi H}}{Y}, \quad \text{as }T\to \infty ,
\end{equation*}
where $X$ and $Y$ are two independent $\mathcal{N}(0,1)$ random
variables.\index{random variables} This coincides with the result of
\cite[Th. 2]{Tanaka2015}.
\end{remark}

\begin{appendix}

\section{Modified Bessel function of the first kind}%
\label{sec:Bessel-function}
In this section we present some properties of the modified Bessel function
of the first kind $I_{\nu }(x)$, which are helpful for our proofs. For
more details on this topic we recommend the book~\cite{Watson}.
Let $\nu >-1$, $x\in \mathbb{R}$. Then the function $I_{\nu }(x)$ could
be defined by the following power series
\cite[Formula~50:6:1]{An-Atlas-of-Functions}:
\begin{equation*}
I_{\nu }(x) = \sum_{j=0}^{\infty }
\frac{(x/2)^{2j+\nu }}{j!\varGamma (j+1+
\nu )}.
\end{equation*}
Note, that if $x$ is negative and $\nu $ is non-integer, then the
function $I_{\nu }(x)$ is complex-valued. However, a function
$I_{\nu }(x)/x^{\nu }$ is always real-valued. This function equals
$2^{-\nu }/\varGamma (1+\nu )$ when $x=0$ and it is even, i.e.
%
\begin{equation}
\label{eq:even} \frac{I_{\nu }(-x)}{(-x)^{\nu }} = \frac{I_{\nu }(x)}{x^{\nu }},
\end{equation}
see~\cite[Formula~50:2:1]{An-Atlas-of-Functions}. For large values
of $x$ the function $I_{\nu }(x)$ has the following asymptotic behavior
\cite[Formula~9.7.1]{Abramovitz-Stegun}:
%
\begin{equation}
\label{eq:I_asymptotic} I_{\nu }(x) = \frac{e^{x}}{\sqrt{2 \pi x}} \biggl( 1 -
\frac{4 \nu
^{2} - 1}{8x} + O \bigl( x^{-2} \bigr) \biggr), \quad x \to
\infty .
\end{equation}

\section{MGFs related to the bivariate normal distribution}%
\label{sec:product_mgf_general}
MGF of the product of two normal random variables $X \stackrel{d}{=}
\mathcal{N}(m_{1}, \sigma _{1}^{2})$, $Y \stackrel{d}{=}\mathcal{N}(m
_{2},\allowbreak  \sigma _{2}^{2})$ with the correlation coefficient $r = \frac{\cov (X,
Y)}{\sigma _{1} \sigma _{2}}$ equals (see e.g.~\cite{Craig})
%
\begin{equation}
\label{eq:product_mgf_general} \mathbb{E} \bigl[ \exp \{ t X Y \} \bigr] = D^{-1/2}
\exp \biggl\{ \frac{  ( m_{1}^{2} \sigma _{2}^{2} + m_{2}^{2}
\sigma _{1}^{2} - 2 r m_{1} m_{2} \sigma _{1} \sigma _{2}   )t^{2} +
2 m_{1} m_{2} t}{2 D} \biggr\} ,
\end{equation}
where
\begin{equation*}
D = \bigl( 1 - (1 + r) \sigma _{1} \sigma _{2} t
\bigr) \bigl( 1 + (1 - r) \sigma _{1} \sigma _{2} t
\bigr).
\end{equation*}

\begin{lemma}
\label{l:product_mgf_for_I+bK}
For two independent normal random variables\index{random variables} $X \stackrel{d}{=}
\mathcal{N}  (m, \sigma ^{2}   )$ and $Y \stackrel{d}{=}
\mathcal{N}(0, 1)$,
%
\begin{align}
\label{eq:product_mgf_for_I+bK} \mathbb{E} \bigl[ \exp \bigl\{ \theta _{1} X Y +
\theta _{2} X^{2} \bigr\} \bigr]
= \bigl[1 - \sigma ^{2} \bigl(\theta _{1}^{2}
+ 2\theta _{2} \bigr) \bigr] ^{-\frac{1}{2}} \exp \biggl\{
\frac{m^{2}  (\theta _{1}^{2} + 2
\theta _{2}  )}{2   [1 - \sigma ^{2}   (\theta _{1}^{2} + 2
\theta _{2}  )  ] } \biggr\} .
\end{align}
\end{lemma}

\begin{proof}
Evidently,
\begin{equation*}
\theta _{1} X Y + \theta _{2} X^{2} = X (
\theta _{1} Y + \theta _{2} X ),
\end{equation*}
where $X \stackrel{d}{=}\mathcal{N}  (m, \sigma ^{2}   )$,
$\theta _{1} Y + \theta _{2} X \stackrel{d}{=}\mathcal{N}  (\theta
_{2} m, \theta _{1}^{2} + \theta _{2}^{2}\sigma ^{2}   )$, and
\begin{equation*}
\cov (X,\theta _{1} Y + \theta _{2} X ) =\theta
_{2}\cov (X, X ) = \theta _{2} \sigma ^{2}.
\end{equation*}
Applying \eqref{eq:product_mgf_general} with $t=1$, $m_{1}=m$,
$\sigma _{1}^{2}=\sigma ^{2}$, $m_{2}=\theta _{2} m$, $\sigma _{2}^{2}=
\theta _{1}^{2} + \theta _{2}^{2}\sigma ^{2}$, and $r=\frac{\theta _{2}
\sigma }{\sqrt{\theta _{1}^{2} + \theta _{2}^{2}\sigma ^{2}}}$, we get
\begin{align*}
& \mathbb{E} \bigl[ \exp \bigl\{ X (\theta _{1} Y + \theta
_{2} X ) \bigr\} \bigr]
\\*
&\quad = D^{-1/2} \exp \biggl\{ \frac{m^{2}  (\theta _{1}^{2} + \theta
_{2}^{2}\sigma ^{2}  ) + \theta _{2}^{2} m^{2} \sigma ^{2} - 2\theta
_{2}^{2} m^{2} \sigma ^{2} + 2 \theta _{2} m^{2}}{2 D} \biggr\}
\\
&\quad = D^{-1/2} \exp \biggl\{ \frac{m^{2}  (\theta _{1}^{2} + 2\theta
_{2}  )}{2 D} \biggr\} ,
\end{align*}
where
\begin{align*}
D &= \biggl(1- \biggl(1+\frac{\theta _{2} \sigma }{\sqrt{\theta _{1}
^{2} + \theta _{2}^{2}\sigma ^{2}}} \biggr)\sigma \sqrt{\theta
_{1} ^{2} + \theta _{2}^{2}
\sigma ^{2}} \biggr)
\\*
&\quad \times \biggl(1+ \biggl(1-\frac{\theta _{2} \sigma }{\sqrt{
\theta _{1}^{2} + \theta _{2}^{2}\sigma ^{2}}} \biggr)\sigma \sqrt{
\theta _{1}^{2} + \theta _{2}^{2}
\sigma ^{2}} \biggr)
\\
&= \bigl(1-\sigma \sqrt{\theta _{1}^{2} + \theta
_{2}^{2}\sigma ^{2}} -\theta _{2}
\sigma ^{2} \bigr) \bigl(1+\sigma \sqrt{\theta _{1}^{2}
+ \theta _{2}^{2} \sigma ^{2}} -\theta
_{2} \sigma ^{2} \bigr)
\\
&= \bigl(1-\theta _{2} \sigma ^{2} \bigr)^{2}
-\sigma ^{2} \bigl(\theta _{1}^{2} + \theta
_{2}^{2}\sigma ^{2} \bigr) = 1 - \sigma
^{2} \bigl(\theta _{1}^{2} + 2\theta
_{2} \bigr),
\end{align*}
whence \eqref{eq:product_mgf_for_I+bK} follows.
\end{proof}
\end{appendix}

\begin{acknowledgement}
The authors are grateful to the referees for their valuable comments and
suggestions.
\end{acknowledgement}

\begin{funding}
The second author acknowledges that the present research is carried
through within the frame and support of the ToppForsk project nr. \gnumber[refid=GS1]{274410}
of the \gsponsor[id=GS1,sponsor-id=501100005416]{Research Council of Norway} with title STORM: Stochastics for
Time-Space Risk Models.
\end{funding}


\end{document}